\documentclass[a4paper,11pt,centertags]{amsart}
\usepackage[latin1]{inputenc}
\usepackage[english]{babel}
\usepackage{cite}
\usepackage{verbatim}
\usepackage{color}
\usepackage{hyperref}
\usepackage{palatino,mathpazo}
\usepackage{amsmath,amsthm}
\usepackage{comment}
\usepackage[top=3.2cm, bottom=3.2cm, left=2.8cm,right=2.8cm]{geometry}
\newcommand{\diesis}{^\#}
\newtheorem{theo}{Theorem}[section]
\newtheorem{lemma}{Lemma}[section]

\theoremstyle{definition}
\newtheorem{definiz}{Definition}[section]
\newtheorem{rem}{Remark}[section]

\numberwithin{equation}{section}

\newcommand{\R}{\mathbb R}

\newcommand{\de}{\partial}
\newcommand{\eps}{\varepsilon}

\DeclareMathOperator{\Sk}{S}

\begin{document}
\title[Stability results for some fully nonlinear eigenvalue
estimates]{Stability results for 
some fully nonlinear eigenvalue estimates} 
\author[F. Della Pietra, N. Gavitone]{
  Francesco Della Pietra and Nunzia Gavitone
}
\address{Francesco Della Pietra \\
Universit\`a degli studi del Molise \\
Dipartimento di Bioscienze e Territorio -
Divisione di Fisica, Informatica e Matematica\\
Via Duca degli Abruzzi \\
86039 Termoli (CB), Italia.
}
\email{francesco.dellapietra@unimol.it}

 \address{
Nunzia Gavitone \\
Universit\`a degli studi di Napoli ``Federico II''\\
Dipartimento di Ma\-te\-ma\-ti\-ca e Applicazioni ``R. Caccioppoli''\\
80126 Napoli, Italia.
}
\email{nunzia.gavitone@unina.it}
\date{\today}
\maketitle
\begin{abstract}
In this paper, we give some stability estimates for the Faber-Krahn
inequality relative to the eigenvalue $\lambda_k(\Omega)$ of the
Hessian operator $\Sk_k$, $1\le k \le n$, in a reasonable bounded
domain $\Omega$. Roughly speaking, we prove that if
$\lambda_k(\Omega)$  is near to $\lambda_k(B)$, where $B$ is a
ball which preserves an appropriate measure of $\Omega$, then, in a
suitable sense, $\Omega$ is close to $B$.
\end{abstract}

\section{Introduction}
In this paper we prove some stability estimates for the eigenvalue
$\lambda_k(\Omega)$ of the $k$-Hessian operator, that has the
variational characterization
\begin{equation*}
\lambda_k(\Omega) = \min\left\{\int_\Omega (-u)
    S_k(D^2u)\,dx,\; u\in \Phi_k^2(\Omega) \text{ and }
    \int_\Omega (-u)^{k+1}\,dx=1 \right\}.
  %  \lambda_k(\Omega)=\min_{\substack{u \in 
  %      \Phi_k^2(\Omega) \\
  %      u \ne 0 }} \displaystyle \frac{\int_{\Omega}(-u)
  %    S_k(D^2u)\,dx}{\int_{\Omega}(-u)^{k+1}\,dx}.
  \end{equation*}
Here $\Omega$ is a bounded, strictly convex, open set of $\R^n$, $n\ge
2$, with $C^2$ boundary, $\Sk_k(D^2u)$, with $1\le k\le n$,  is the
$k$-th elementary symmetric function of the eigenvalues of $D^2 u$
with $u\in C^2(\Omega)$, and $\Phi_k^2(\Omega)$ denotes the class of
the admissible functions for $\Sk_k$, the so-called $k$-convex
functions (see Section 2 for the precise definitions). Notice that
$S_1(D^2u)=\Delta u$, the Laplacian operator, while $\Sk_n(D^2u)=\det
D^2 u$, the Monge-Amp\`ere operator. 

It is known that, under suitable assumptions on
$\Omega$, for this kind of operators a Faber-Krahn inequality holds,
that is the eigenvalue $\lambda_k(\Omega)$ attains its minimum value
on the ball $\Omega_{k-1}^*$, which preserves an appropriate curvature
measure of $\Omega$, the $(k-1)$-th quermassintegral:
\begin{equation}
 \label{introfksk}
\lambda_k(\Omega) \ge \lambda_k(\Omega^*_{k-1}), \quad 1\le k\le n
\end{equation}
(see \cite{bt07,ga09}). Moreover, for $k=n$,
$\lambda_n(\Omega)$ is also bounded from above by 
$\lambda_n(\Omega^*_0)$, with $|\Omega^*_0|=|\Omega|$ (see
\cite{bntpoin}). For sake of completeness, we
recall that in the case of Neumann boundary condition, for $k=1$, the
reverse inequality in \eqref{introfksk} holds (see \cite{w56,sz54},
\cite{chdb12} and \cite{brchtr09,dpg2} for related results).

In \cite{dpg5} we give some stability estimates of
\eqref{introfksk}, proving that
\begin{equation*}
  %\label{anpask}
    \frac{\lambda_k(\Omega) -
      \lambda_k(\Omega^*_{k-1})}{\lambda_k(\Omega)} \le C_{n,k}
    \frac{|\Omega^*_k|-|\Omega|}{|\Omega^*_k|}, \quad 1\le k\le n-1,
  \end{equation*}
for some constant $C_{n,k}$ depending only on $n$ and $k$, while for
$k=n$,
\begin{equation*}
   %\label{anpama}
  \frac{\lambda_n(\Omega) -
    \lambda_n(\Omega^*_{n-1})}{\lambda_n(\Omega)} \le C_n
  \frac{|\Omega^*_{n-1}|-|\Omega|}{|\Omega^*_{n-1}|},
\end{equation*}
where $C_n$ which depends only on $n$. Roughly speaking, such
inequalities state that if $\Omega$ is close to a ball with respect
the $L^1$ norm, then their corresponding eigenvalues are near. Such
result is in the spirit of a well-known result due to Payne and
Weinberger for the Laplace operator (see \cite{pw61}), and given in
\cite{bnt10} for the $p$-Laplace operator (see also \cite{nitsch12}
for the best constant in the case $p=2$, and \cite{dpgtors} for
the anisotropic case).

Viceversa, the aim of this paper is to prove some stability results
which ensure  that if $\lambda_k(\Omega)$ is near to
$\lambda_k(\Omega^*_{k-1})$, then, in an appropriate sense, $\Omega$ is
close to a suitable ball of $\R^n$ (see Section 2 for the precise
statements). 

There are several contributions in this direction, for the first
eigenvalue of the Laplacian operator (see \cite{melas92},
\cite{hana94}) or, more generally, for the $p$-Laplacian
(see \cite{bhatta01}, \cite{fmp09}). In such papers, depending on the
assumptions on $\Omega$, suitable notions of the distance between the
set $\Omega$ and a ball are considered. In particular, under the
convexity assumption on the domain, it seems natural to take into
account the Hausdorff distance (see \cite{melas92}), while, in a more
general setting, such notion is replaced by the so-called Fraenkel
asymmetry (see \cite{bhatta01}, \cite{fmp09}). Both arguments are
considered in \cite{hana94}. 

Dealing with convex sets, our aim is to prove some stability result
for Hessian operators in the spirit of the results given in
\cite{melas92,hana94}. 
In particular, we prove that for a strictly convex, smooth domain
$\Omega$, such that
\[
\lambda_k(\Omega) \le \lambda_k(\Omega^*_{k-1})(1+\eps),
\]
for some $\eps>0$ sufficiently small, then there exist two balls
$B_{r_\Omega}$ and $B_{R_\Omega}$ such that $B_{r_\Omega}\subseteq
\Omega \subseteq B_{R_\Omega}$ and two suitable asymmetry coefficients
of $\Omega$ with respect to $\Omega^*_{k-1}$ vanish when $\eps$ goes
to zero. This will imply that the Hausdorff asymmetry of $\Omega$ is
close to zero (see Section 2.4 for the precise definitions). 

The paper is organized as follows. In Section 2, we recall some basic
definitions of convex geometry, and the properties of symmetrization
for quermassintegrals. Moreover, we summarize some useful results on
the eigenvalue problem for Hessian operators. In sections 3 and 4 we
state and prove the main results. We distinguish the case of the
Monge-Amp\`ere operator (see Section 3) from the case of $\Sk_k$,
$1\le k\le n-1$. Our approach makes use of a quantitative version of a
suitable isoperimetric inequality and a symmetrization for
quermassintegral technique.

\section{Notation and preliminaries}
Throughout the paper, we will denote with $\Omega$ a set of
$\R^n$, $n\ge 2$ such that
\begin{equation}
  \label{ipomega}
    \Omega\text{ is a bounded, strictly convex, open set with }C^2
    \text{ boundary}.
  \end{equation}
  By strict convexity of $\Omega$ we mean that the Gauss curvature is
  strictly positive at every point of $\de \Omega$.
  
Given a function $u \in C^2(\Omega)$, we denote by $\lambda(D^2u)=
(\lambda_1,\lambda_2, \ldots,\lambda_n)$ the vector of the 
eigenvalues of $D^2u$.  The $k$-Hessian operator $\Sk_k(D^2u)$, with
$k=1,2,\ldots,n$, is
\begin{equation*}
  \Sk_k(D^2u)=\sum_{i_1<i_2<\cdots<i_k} \lambda_{i_1} \cdot
  \lambda_{i_2} \cdots \lambda_{i_k}.
\end{equation*}
Hence $\Sk_k(D^2u)$ is the sum of all $k \times k$ principal minors of
the matrix $D^2u$.

The $k$-Hessian operator can be written also in divergence form,
that is
\begin{equation*}
  \label{div}
  \Sk_k(D^2u)=\frac{1}{k}\sum_{i,j=1}^n (\Sk_k^{ij}u_i)_j,
\end{equation*}
where $S_k^{ij} = \frac{\partial S_k(D^2u)}{\partial u_{ij}}$ (see
for instance \cite{trudi1}, \cite{trudi2}, \cite{wangeigen}).

Well known examples ok $k$-Hessian operators are $\Sk_1(D^2u)=\Delta
u$, the Laplace operator, and $\Sk_n(D^2u)=\det(D^2u)$, the
Monge-Amp\`ere operator.

It is well-known that $\Sk_1(D^2u)$ is elliptic. This property is not
true in general for $k>1$. As matter of fact, the $k$-Hessian operator
is elliptic when it acts on the class of the so-called
$k$-convex function, defined below.
\begin{definiz}
Let $\Omega$ be as in \eqref{ipomega}. A function $u \in C^2(\Omega)$
is called a $k$-convex function (strictly 
$k$-convex) in $\Omega$ if
\begin{equation*}
\Sk_j(D^2u)\geq 0 \text{ }(>0) \quad \text{for }j=1,
\ldots, k.
\end{equation*}
We denote the class of $k$-convex functions in $\Omega$ such
that $u \in C^2(\Omega)\cap C(\bar{\Omega})$ and  $u=0$ on $\partial
\Omega$ by $\Phi^2_k(\Omega)$. 
\end{definiz}
Clearly, the set $\Phi_n^2(\Omega)$
coincides with the set of the convex and $C^2(\Omega)$ functions
vanishing on $\de \Omega$.

If we define with $\Gamma_k$ the following convex open cone
\begin{equation*}
  \Gamma_k=\{ \lambda \in \R^n : S_1(\lambda)>0, S_2(\lambda)>0,
  \ldots, S_k(\lambda)>0\},
\end{equation*}
in \cite{ivo} it is proven that $\Gamma_k$ is the cone of
ellipticity of $\Sk_k$. Hence the $k$-Hessian operator is elliptic
with respect to the $k$-convex functions.

By definition, it follows that the $k$-convex functions are
subharmonic in $\Omega$ and then negative in $\Omega$ if zero on
$\de\Omega$.

We go on by recalling some definitions of convex geometry which will
be largely used in next sections. Standard references for this topic
are \cite{bz}, \cite{schn}.

\subsection
{Quermassintegrals and the Aleksandrov-Fenchel inequalities} 

Let $K$ be a convex body, and let be $\rho>0$. We denote by $|K|$ the
Lebesgue measure of $K$, by $P(K)$ the perimeter of $K$ and by
$\omega_n$ the measure of the unit ball in $\R^n$.

The well-known Steiner formula for the Minkowski sum is
\[
|K+\rho B_1| =\sum_{i=0}^{n} \binom{n}{i} W_i(K) \rho^i.
\]
The coefficient $W_i(K)$, $i=0,\ldots,n$, is known as the $i$-th
quermassintegral of $K$. Some special cases are $W_0(K)=|K|$,
$nW_1(K)=P(K)$, $W_n(K)=\omega_n$. If $K$ has $C^2$ boundary, with
nonvanishing Gaussian curvature, the quermassintegrals can be related
to the principal curvatures of $\de K$. Indeed, in such a case
\[
W_i(K)=\frac 1 n \int_{\de K} H_{i-1} d \mathcal H^{n-1}, \quad
i={1,\ldots n}.
\]
Here $H_j$ denotes the $j$-th normalized elementary symmetric function
of the principal curvatures $\kappa_1,\ldots,\kappa_{n-1}$ of $\de K$,
that is $H_0=1$ and
\[
H_j= \binom{n-1}{j}^{-1} \sum_{1\le i_1\le \ldots \le i_j\le n-1}
\kappa_{i_1}\cdots \kappa_{i_j},\quad j={1,\ldots,n-1}.
\]

An immediate computation shows that if $B_R$ is a ball of radius $R$,
then
\begin{equation}
  \label{querball}
  W_i(B_R)= \omega_n R^{n-i}, \quad i=0,\ldots,n.
\end{equation}
Moreover, the $i$-th quermassintegral, $0\le i \le n$, rescales as
\[
W_{i}(tK)=t^{n-i}W_i(K), \quad t>0.
\]

The Aleksandrov-Fenchel inequalities state that
\begin{equation}
  \label{afineq}
\left( \frac{W_j(K)}{\omega_n} \right)^{\frac{1}{n-j}} \ge \left(
  \frac{W_i(K)}{\omega_n} \right)^{\frac{1}{n-i}}, \quad 0\le i < j
\le n-1,
\end{equation}
where the inequality is replaced by an equality if and only if $K$ is
a ball.

In what follows, we use the Aleksandrov-Fenchel inequalities for
particular values of $i$ and $j$. If $i=1$, and $j=k-1$, we have that
\begin{equation}\label{af-2}
W_{k-1}(K) \ge \omega_n^{\frac{k-2}{n-1}} n^{-\frac{n-k+1}{n-1}}
P(K)^{\frac{n-k+1}{n-1}}, \quad
3\le k \le n.
\end{equation}
When $i=0$ and $j=1$, we have the classical isoperimetric inequality:
\[
P(K) \ge n \omega_n^{\frac 1 n} |K|^{1-\frac 1 n}.
\]
Moreover, if $i=k-1$, and $j=k$, we have
\[
W_k(K) \ge \omega_n^{\frac{1}{n-k+1}} W_{k-1}(K)^{\frac{n-k}{n-k+1}}.
\]

It can be also shown a derivation formula for quermassintegral of
level sets of a function $u \in \Phi_k^2(\Omega)$ (see
\cite{reilly}):
\begin{equation}
  \label{reillyder}
  - \frac{d}{dt} W_{k}(\Omega_t)= \frac{n-k}{n}
  \int_{\Sigma_t}H_{k}(\Sigma_t) |Du|^{-1} d \mathcal H^{n-1}, 
\end{equation}
where $\Sigma_t$ is the boundary of $\Omega_t=\{ -u > t \}$.
Moreover, we recall the following equality (see \cite{reilly} again):
\begin{equation}
\label{introintcurv}
\int_{\Omega_t}\Sk_k (D^2u)dx = \frac{1}{k} \int_{\Sigma_t} H_{k-1} |Du|^k
d\mathcal H^{n-1}.
\end{equation}

\subsection{Rearrangements for quermassintegrals}
Now we recall some basic facts on rearrangements for
quermassintegrals. For an exhaustive treatment of the properties of
such rearrangements we refer the reader, for example, to \cite{ta81}, 
\cite{tso}, \cite{trudiso}.

Let $1\le k\le n$, and denote by $\Omega^*_{k-1}$ the ball centered at the
origin and with the same $W_{k-1}$-measure than $\Omega$, that is
$W_{k-1}(\Omega^*_{k-1})=W_{k-1}(\Omega)$.

The $(k-1)$-symmetrand of a function $u \in \Phi_{k}(\Omega)$,
$k=1,\ldots,n$, is the radially symmetric increasing function
$u^*_{k-1}$, defined in the ball $\Omega^*_{k-1}$, which preserves
the $W_{k-1}$-measure of the level sets of $u$. More precisely, we
have that, for $x\in\Omega^*_{k-1}$, 
\begin{equation*}
u_{k-1}^{*}(x)=-\inf \left\{ t \ge 0 \colon W_{k-1}(\Omega_t) \le
\omega_n\left|x\right|^{n-k+1},\; Du \ne 0 \text{ on } \Sigma_t\right\},
\end{equation*}
where $\Sigma_{t} = \partial \Omega_t=\left\{x \in \Omega\colon
-u(x)>t\right\}$, with $t\ge 0$.

We stress that, for $k=1$, $u^{*}_{0}(x)$ coincides with
the classical Schwarz symmetrand of $u$, while for $k=2$,
$u^*_1(x)$ is the rearrangement of $u$ which preserves the
perimeter of the level sets of $u$. 

Denoting with $R$ the radius of $\Omega^*_{k-1}$, the following
statements hold true (see \cite{tso,trudiso}): 
\begin{enumerate}
    \item writing $u_{k-1}^{*}(x) = \rho(r)$ for
      $r=\left|x\right|$, we have $\rho(0)=\min_{\Omega}(u)$ and
      $\rho(R)=0$,
    \item $\rho(r)$ is a negative and increasing function on
      $\left[0,R\right]$,
    \item $\rho(r) \in C^{0,1}(\left[0,R\right])$ and moreover $0 \le
      \rho'(r)\leq \sup_{\Omega}\left|Du\right|$ almost everywhere.
\end{enumerate}

If the function $u$ has convex level sets, the Aleksandrov-Fenchel
inequalities \eqref{afineq} imply that
$|\{ -u>t \}|\le |\{-u^*_{k-1}>t\}|$ and then, for any $p\ge 1$,
\begin{equation*}
  \label{norme}
\left\|u\right\|_{L^p(\Omega)} \le
\left\|u_{k-1}^{*}\right\|_{L^p(\Omega^*_{k-1})},
\end{equation*}
while, by property $(1)$,
\[
\|u\|_{{L^\infty(\Omega)}} =\|u^*_{k-1}\|_{{L^\infty(\Omega^*_{k-1})}}.
\]
Now, it is possible to define the following functional associated to
the $k$-Hessian operator, known as $k$-Hessian integral: 
\begin{equation*}
I_k\left[u,\Omega \right]=\int_{\Omega} (-u)\Sk_k(D^2u)
 \, dx.
\end{equation*}

In the radial case the Hessian integrals can be defined as
follows:
\begin{equation*}
I_{k}\left[ u^{*}_{k-1}, \Omega^*_{k-1}\right]= 
\binom{n}{k} \omega_n \int_0^R f^{k+1}\big(\omega_n
r^{n-k+1}\big) r^{n-k} \, dr
\end{equation*}
where $f\big(\omega_n \left|x\right|^{n-k+1}\big) = \left|\nabla
u^{*}_{k-1}(x)\right|$.

Finally we recall that for the Hessian integrals the following
extension of P\'olya-Szeg\"o principle holds (see
\cite{trudiso,tso}):
\begin{equation}
\label{pol} 
I_{k}\left[u, \Omega\right]\geq
I_{k}\left[u^{*}_{k-1}, \Omega^*_{k-1}\right], \quad p \geq 1.
\end{equation}
\subsection{Eigenvalue problems for $\Sk_k$}
In this subsection we give a quick review on the main properties of 
eigenvalues and eigenfunctions of the $k$-Hessian operators, namely
the couples $(\lambda,u)$ which solve
\begin{equation}
  \label{autsk}
  \left\{
    \begin{array}{ll}
      S_k(D^2u)=\lambda (-u)^k &\text{in } \Omega,\\
      u=0 &\text{on } \de\Omega.
    \end{array}
  \right.
\end{equation}

The following existence result holds (see \cite{lionsma} for $k=n$,
and \cite{wangeigen}, \cite{geng} in the general case):
\begin{theo}
%  \label{defaut}
  Let $\Omega$ as in \eqref{ipomega}. Then, there exists a positive
  constant $\lambda_k(\Omega)$ depending only on $n,k$, and $\Omega$,
  such that problem (\ref{autsk}) admits a 
  solution $u \in C^2(\Omega)\cap C^{1,1}(\overline{\Omega})$,
  negative in $\Omega$, for
  $\lambda=\lambda_k(\Omega)$ and $u$ is unique up to positive
  scalar multiplication. Moreover, $\lambda_k(\Omega)$ has the following
  variational characterization:
  \begin{equation*}
    % \label{carvar}
    \lambda_k(\Omega)=\min_{\substack{u \in
        \Phi_k^2(\Omega) \\
        u \ne 0 }} \displaystyle \frac{\int_{\Omega}(-u)
      S_k(D^2u)\,dx}{\int_{\Omega}(-u)^{k+1}\,dx}.
  \end{equation*}
\end{theo}

As matter of fact, if $k<n$ the above theorem holds under a more
general assumption on $\Omega$, namely requiring that $\Omega$ is
strictly $k$-convex (see \cite{wangeigen}, \cite{geng}).

As matter of fact, we observe that if $k=1$, or $k=n$,
$\lambda_k(\Omega)$ coincides respectively with the first eigenvalue
of the Laplacian operator, or with the eigenvalue of the
Monge-Amp\`ere operator.

A simple but useful property of the eigenvalue $\lambda_k(\Omega)$
is that it rescales as
\begin{equation}
\label{risc_palla}
\lambda_k(t\Omega)=t^{-2k}\lambda_k(\Omega),\quad t>0.  
\end{equation}

If $k=1$, the well-known Faber-Krahn inequality states that
\[
\lambda_1(\Omega) \ge \lambda_1(\Omega\diesis),
\]
where $\Omega\diesis$ is the ball centered at the origin
with the same Lebesgue measure of $\Omega$. Moreover, the equality
holds if $\Omega=\Omega\diesis$. 

In \cite{bt07}, \cite{ga09} it is proved that if $k=n$ and $\Omega$ is
a bounded strictly convex open set, then 
\begin{equation*}
 % \label{fkdet}
\lambda_n(\Omega) \ge \lambda_n(\Omega_{n-1}^*).
\end{equation*}
In general, in \cite{ga09} it is proven that if $\Omega$ is
a strictly convex set such that the eigenfunctions have convex level
sets, then, for $2\le k \le n$,
\begin{equation}
  \label{fksk}
\lambda_k(\Omega) \ge \lambda_k(\Omega^*_{k-1}).
\end{equation}
\subsection{Asymmetry measures and isoperimetric deficit}
A purpose of this paper is to prove that the difference
between the two sides in \eqref{fksk} controls the exterior and interior
deficiencies, defined as follows (see also \cite{hana94} for $k=1$).
Given $\Omega$ bounded nonempty domain of $\R^n$, denoted by $R$ the
radius of the ball $\Omega^*_{k-1}$, then the exterior and interior
$k$-deficiency of $\Omega$ are, respectively, the nonnegative numbers
\begin{equation}
  \label{defe}
 D_{k}(\Omega)= \frac{R_\Omega}{R}-1,\quad d_k(\Omega)= 1-\frac{r_\Omega}{R},
\end{equation}
where $r_\Omega$ and $R_\Omega$ are the inradius and the circumradius
of $\Omega$. Such numbers give a measure of the asymmetry of $\Omega$
with respect to the ball with the same $(k-1)$-quermassintegral than
$\Omega$. Furthermore, the deficiency of $\Omega$ is
\begin{equation*}
  \Delta(\Omega)=\frac{R_\Omega}{r_\Omega}-1.
\end{equation*}

In order to have a measure of the asymmetry of $\Omega$ in terms of
the Hausdorff distance $d$, we define the following coefficient:
\begin{equation}
  \label{haus}
  \delta_H(\Omega) = \inf \{d(\Omega, \Omega^*_{n-1}+x_0),\; x_0\in \R^n\}.
\end{equation}
We refer to $\delta_H$ as the Hausdorff asymmetry of
$\Omega$.

In the class of convex sets, it is possible to obtain some stability
estimates for the Aleksandrov-Fenchel inequalities \eqref{afineq}.
More precisely, if $s$ denotes the Steiner point of
$\Omega$, then in \cite{grsc} it has been proved that
\begin{equation}
  \label{gshaus}
  \begin{array}{l}
d(\Omega,\Omega^*_{n-1}+s)^{(n+3)/2} \le \bar C 
\frac{P(\Omega)^{(n^2-3)/2}}{|\Omega|^{(n+3)(n-2)/2}} \left[
  \left(\frac{P(\Omega)}{n\omega_n}\right)^n -
  \left(\frac{|\Omega|}{\omega_n}\right)^{n-1} \right],\\
  d(\Omega,\Omega^*_{n-1}+s)^{(n+3)/2}
  \le \bar C_1
  \frac{W_{n-2}(\Omega)W_{n-1}(\Omega)^{\frac{n-1}{2}}}{W_{k-1}(\Omega)^{n-k}} 
  \left( \frac{W_k(\Omega)^{n-k+1}}{\omega_n} -W_{k-1}(\Omega)^{n-k}
  \right),
  \end{array}
\end{equation}
where $\bar C,\bar C_1$ are two constants depending only on $n$, which can be
explicitly determined. These estimates justify the definition of
$\delta_H$.

As matter of fact, in \cite{grsc} it is observed that, the inequalities
\eqref{gshaus} can be rewritten as a Bonnesen-type inequality in terms
of the inradius $r_\Omega$ and the circumradius $R_\Omega$ of $\Omega$:
\begin{equation}
  \label{gs}
  \begin{array}{l}
(R_\Omega-r_\Omega)^{(n+3)/2} \le \tilde C
\frac{P(\Omega)^{(n^2-3)/2}}{|\Omega|^{(n+3)(n-2)/2}}\left[
  \left(\frac{P(\Omega)}{n\omega_n}\right)^n -
  \left(\frac{|\Omega|}{\omega_n}\right)^{n-1} \right],\\
(R_\Omega-r_\Omega)^{(n+3)/2} \le \tilde C_1
  \frac{W_{n-2}(\Omega)W_{n-1}(\Omega)^{\frac{n-1}{2}}}{W_{k-1}(\Omega)^{n-k}} 
  \left( \frac{W_k(\Omega)^{n-k+1}}{\omega_n} - W_{k-1}(\Omega)^{n-k}
  \right).
  \end{array}
\end{equation}

\section{The case of the Monge-Amp\`ere operator}
In this section we consider the eigenvalue problem for the
Monge-Amp\`ere operator,
\begin{equation}\label{eig_eq}
  \left\{
    \begin{array}{ll}
      -\det D^2 u= \lambda (-u)^{n} &
      \text{in }\Omega, \\ [.2cm]
      u = 0 & \text{on }\de\Omega,
    \end{array}
  \right.
\end{equation}
and we prove the first stability result, stated below. 
\begin{theo}
  \label{main}
  Let $\Omega \subset \R^n$ be as in \eqref{ipomega} such that
  \begin{equation}
    \label{small}
    \lambda_n(\Omega) \le (1+\eps)\lambda_n(\Omega_{n-1}^*), 
  \end{equation}
  where $\eps>0$ is sufficiently small and $\Omega_{n-1}^*$ is the
  ball such that $W_{n-1}(\Omega)=W_{n-1}(\Omega_{n-1}^*)$. Then, if
  $\delta_H(\Omega)$ is the Hausdorff asymmetry \eqref{haus}, it holds that
  \begin{equation*}
    \label{maindelta}
   \delta_H(\Omega)\le C_n \eps^{\frac{1}{(n+1)(n+3)}},
 \end{equation*}
 where $C_n$ is a constant which depend only on $n$. Moreover,
 denoting by $d_n(\Omega)$ and $D_n(\Omega)$, respectively, the
 interior and exterior $n$-deficiency of $\Omega$  as in
 \eqref{defe}, we have the following:  
  \begin{enumerate}
\item if $n=2$, then
  \begin{equation}
    \label{defdet2}
    d_2(\Omega) \le C_2 
    \sqrt[6]\eps, \quad
 D_2(\Omega) \le C_2 \sqrt[12]\eps,
    \end{equation}
    where $C_2$ denotes a positive constant which depends only on
    the dimension $n=2$. 
  \item If $n\ge 3$, then
    \begin{equation}
      \label{defdetn}
   d_n(\Omega) \le C_n \eps^{\frac{1}{2n+2}},
      \quad
    D_n(\Omega)\le C_n \eps^{\frac{1}{(n+1)(n+3)}},
    \end{equation}
        where $C_n$ depends only on $n$.
  \end{enumerate}
\end{theo}
\begin{rem}
  \label{remeq}
 The estimates \eqref{defdet2} and \eqref{defdetn} can be read as
   \[
    \frac{P(\Omega)-P(B_{r_\Omega})}{P(\Omega)} \le C_2 \sqrt[6]\eps,
    \qquad\quad
    \frac{P(B_{R_\Omega})-P(\Omega)}{P(\Omega)} \le C_2 \sqrt[12]\eps,
    \]
and
\begin{equation*}
  \frac{W_{n-1}(\Omega) - W_{n-1}(B_{r_\Omega})}{W_{n-1}(\Omega)}
  \le C_n \eps^{\frac{1}{2n+2}},
  \qquad
  \frac{W_{n-1}(B_{R_\Omega}) - W_{n-1}(\Omega)}{W_{n-1}(\Omega)} \le
  C_n \eps^{\frac{1}{(n+1)(n+3)}},
\end{equation*}
in the spirit of the stability result contained in \cite{melas92}.
\end{rem}
To prove the Theorem, we need some preliminary lemmas.
For $\delta\ge 0$, we denote 
\[
\Omega_\delta = \{x\in \Omega \colon -u > \delta \}. 
\]
In the following result  we estimate  $W_{n-1}(\Omega_{\delta})$ in
term of $W_{n-1}(\Omega)$.  
\begin{lemma}\label{lemma1}
 Under the hypotheses of Theorem \ref{main}, if $u$ is the 
eigenfunction of \eqref{eig_eq} such that $\|u\|_{L^{n+1}(\Omega)}=1$, then for
any $\delta$ such that $0<\delta<\frac 1 2 |\Omega|^{-\frac{1}{n+1}}$, we have
  \[
  W_{n-1}(\Omega_\delta)\ge W_{n-1}(\Omega) (1-\max\{\eps, 2\delta
  |\Omega|^{\frac{1}{n+1}}\}).
  \]
  \begin{proof}
    For $\delta>0$, we compute the Rayleigh quotient of the function
    $\phi=u+\delta$ in $\Omega_\delta$. Then,
    \begin{equation}
      \label{udelta}
          \lambda_n(\Omega_\delta) \le \dfrac{\int_{\Omega_\delta}
      (-\phi)\det D^2 \phi\,dx }{ \int_{\Omega_\delta}
      (-\phi)^{n+1}\,dx} = \dfrac{\int_{\Omega_\delta}
      (-u-\delta)\det D^2 u \,dx }{ \int_{\Omega_\delta}
     (-u-\delta)^{n+1}\,dx}. 
    \end{equation}
   Moreover, being $u$ a solution of \eqref{eig_eq} with
   $\lambda=\lambda_n(\Omega)$, we get, by H\"older inequality, and
   recalling that $\int_\Omega (-u)^{n+1} dx =1$, that   
   \[
   \begin{split}
     %\label{eq:1}
     \int_{\Omega_\delta}(-u-\delta)\det D^2 u \,dx
     &= \lambda_n(\Omega) \int_{\Omega_\delta} (-u-\delta)(-u)^n\,dx \\
     &\le \lambda_n(\Omega) \left(\int_{\Omega_\delta}
       (-u-\delta)^{n+1}\,dx\right)^{\frac{1}{n+1}}
     \left(\int_{\Omega_\delta}
       (-u)^{n+1}\,dx\right)^{\frac{n}{n+1}} \\
     &\le \lambda_n(\Omega) \left(\int_{\Omega_\delta}
       (-u-\delta)^{n+1}\,dx\right)^{\frac{1}{n+1}}. 
   \end{split}
   \]
   Hence, combining the above estimate with \eqref{udelta} it follows
   that
   \begin{equation}
     \label{udelta2}
   \lambda_n(\Omega_\delta) \le \lambda_n(\Omega) \left(\int_{\Omega_\delta}
       (-u-\delta)^{n+1}\,dx \right)^{-\frac{n}{n+1}}    
   \end{equation}
   On the other hand, by Minkowski inequality and choosing
   $\delta<\frac 1 2 |\Omega|^{-\frac{1}{n+1}}$, we obtain
   that 
   \[
   \begin{split}
   \left(\int_{\Omega_\delta} (-u-\delta)^{n+1}\,dx\right)^{\frac{1}{n+1}}
   &\ge \left( \int_{\Omega_\delta} (-u)^{n+1} \,dx \right)^{\frac
     {1}{n+1}} - \left( \int_{\Omega_\delta} \delta^{n+1} \,dx \right)^{\frac
     {1}{n+1}} \\ &\ge \left(1- \int_{\Omega\setminus\Omega_\delta}
     \delta^{n+1} \,dx \right)^{\frac{1}{n+1}}- \delta
   |\Omega_\delta|^{\frac{1}{n+1}} \\
   &= \left(1-
     \delta^{n+1} \big(|\Omega|-|\Omega_\delta|\big)
   \right)^{\frac{1}{n+1}}- \delta |\Omega_\delta|^{\frac{1}{n+1}} \\
   &\ge 1-\delta\big( |\Omega| - |\Omega_\delta| \big)^{\frac{1}{n+1}}
   - \delta |\Omega_\delta|^{\frac{1}{n+1}} \ge 1-2\delta
   |\Omega|^{\frac {1}{n+1}}.
   \end{split}
   \]
   Hence, from \eqref{udelta2}, \eqref{small} and Faber-Krahn
   inequality it follows that 
   \[
  \lambda_n((\Omega_\delta)^*_{n-1}) \le \lambda_n(\Omega_\delta) \le
  (1+\eps) \lambda_n(\Omega^*_{n-1}) \big( 1-2\delta |\Omega|^{\frac
    {1}{n+1}} \big)^{-n} 
  \]
  which implies, by \ref{risc_palla}, that
  \begin{equation}
    \label{udelta3}
   \left(\frac{W_{n-1}(\Omega_\delta)}{W_{n-1}(\Omega)}\right)^{2n}
   =
   \frac{\lambda_n(\Omega^*_{n-1})}{\lambda_n((\Omega_\delta)^*_{n-1})}  
   \ge
    \frac{\big( 1-2\delta |\Omega|^{\frac
    {1}{n+1}} \big)^n}{1+\eps},
\end{equation}
where we used that the balls $\Omega^*_{n-1}$ and
$(\Omega_\delta)^*_{n-1}$ preserve, respectively, the $(n-1)$-th
quermassintegral of $\Omega$ and $\Omega_\delta$. Hence, by
\eqref{udelta3} we get that 
\[
\frac{W_{n-1}(\Omega_\delta)}{W_{n-1}(\Omega)}
\ge \left( 1 - \frac{\eps+ 2 \delta |\Omega|^{\frac
      {1}{n+1}}}{1+\eps} \right)^{1/2} \ge
1-\max\{\eps, 2\delta |\Omega|^{\frac{1}{n+1}}\},
\]
obtaining the thesis.
 \end{proof}
\end{lemma}
The second lemma we need is the following.
\begin{lemma}
Under the hypotheses of Theorem \ref{main}, if $\Omega_t=\{-u>t\}$,
and $u$ is the eigenfunction of \eqref{eig_eq} in $\Omega$ such that
$\|u\|_{L^{n+1}(\Omega)}=1$, then
\begin{equation}
  \label{eq:3}
  \int_0^{+\infty} t^n \big( W_{n-1}(\Omega_t)^n
- \omega_n^{n-1}|\Omega_t|\big) dt \le \frac{\omega_n^{n-1}}{n+1} \eps. 
\end{equation}
\end{lemma}
\begin{proof}
  We consider the difference of the eigenvalues related to the
  sets $\Omega$ and $\Omega^*_{n-1}$. Choosing $u$ as the normalized
  eigenfunction of \eqref{eig_eq} in $\Omega$, using the
  P\'olya-Szeg\"o principle \eqref{pol}
  \begin{equation}
    \label{eq:2}
  \begin{split}
  \lambda_n(\Omega) - \lambda_n(\Omega^*_{n-1}) & \ge
    \int_\Omega (-u)\det D^2 u\,dx -
\frac{\int_{\Omega^*_{n-1}}(-u_{n-1}^* )\det D^2 
  u_{n-1}^*\,dx }{\int_{\Omega^*_{n-1}}(-u_{n-1}^* )^{n+1} \,dx}  \\[.2cm]
& \ge \frac{\int_{\Omega^*_{n-1}}(-u_{n-1}^* )\det D^2 
  u_{n-1}^*\,dx }{\int_{\Omega^*_{n-1}}(-u_{n-1}^* )^{n+1} \,dx}
\left(\int_{\Omega^*_{n-1}}(-u_{n-1}^* )^{n+1}
    \,dx - 1 \right) \\[.2cm]
&\ge \lambda_n(\Omega^*_{n-1})\left(\int_{\Omega^*_{n-1}}(-u_{n-1}^*
  )^{n+1} \,dx -1 \right).
  \end{split}
  \end{equation}
  On the other hand, recalling that $u$ has normalized $L^{n+1}$ norm,
  the coarea formula and an integration by parts give that
  \[
  \begin{split}
  \int_{\Omega^*_{n-1}}(-u_{n-1}^*)^{n+1} \,dx -1 &= 
  (n+1)\int_0^{+\infty} t^n \big(
  |\{-u^*_{n-1}>t\}|-|\{-u>t\}|\big) dt \\
  &= (n+1) \omega_n^{1-n} \int_0^{+\infty} t^n \big(
    W_{n-1}(\Omega_t)^n - \omega_n^{n-1}|\Omega_t|\big).
  \end{split}
  \]
  Hence, joining \eqref{eq:2} with the above equality, and using
  \eqref{small} we obtain that
  \[
  \int_0^{+\infty} t^n \big( W_{n-1}(\Omega_t)^n
  - \omega_n^{n-1}|\Omega_t|\big) dt \le
  \frac{\omega_n^{n-1}}{n+1} \eps,
  \]
  that is the thesis.
\end{proof}
Last lemma plays a key role in order to obtain that the constant $C_n$
involved in (1) and (2) in Theorem \ref{main} is independent on
$\Omega$. 
\begin{lemma}
\label{boundmis}
Under the hypotheses of Theorem \ref{main}, it holds that
\begin{equation*}
%\label{mis}
|\Omega| \ge \tilde C_n \left[W_{n-1}(\Omega)\right]^n,
\end{equation*}
where $\tilde C_n$ denotes a positive constant depending only on $n$.
\end{lemma}
\begin{proof}
Let $u$ be an eigenfunction of \eqref{eig_eq} corresponding to the
eigenvalue $\lambda=\lambda_n(\Omega)$. Then
\begin{equation}
\label{eq}
\det D^2u = \lambda (-u)^n \quad \text{ in }\Omega.
\end{equation} 
Integrating both sides in \eqref{eq} on the level set
$\Omega_t=\{-u>t\}$, and denoting by $\Sigma_t=\de \Omega_t=\{-u=t\}$,
and using the H\"older inequality we have
\begin{multline}
\label{sinistra}
\int_{\Omega_t}\det D^2u\,dx = \frac{1}{n} \int_{\Sigma_t} H_{n-1}
|Du|^n d\mathcal H^{n-1}
\ge \\ \ge \frac{1}{n} \frac{\left(\int_{\Sigma_t} H_{n-1} d\mathcal
      H^{n-1}\right)^{n+1}}{ \left(\int_{\Sigma_t}  H_{n-1} |Du|^{-1} d\mathcal
  H^{n-1}\right)^{n}} = \frac{1}{n}
\frac{(n\omega_n)^{n+1}}{\left(-\frac{d}{dt}W_{n-1}(\Omega_t)\right)^n}. 
\end{multline}
Last inequality follows by the H\"older inequality and the properties of
quermassintegrals. Moreover, being $|\Omega_t|\le |\Omega|$, we have
\begin{equation}
\label{destra}
\left(\int_{\Omega_t}(-u)^n\,dx \right)^\frac{1}{n} \le
|\Omega|^{\frac 1 n} \|u\|_{L^{\infty}(\Omega)}.
\end{equation}
Putting togheter \eqref{sinistra} and \eqref{destra}, by \eqref{eq} we
get
\[
-\frac{d}{dt}W_{n-1}(\Omega_t) \ge n \omega_n^{1+\frac{1}{n}}
\lambda^{-\frac{1}{n}}|\Omega|^{-\frac 1 n} \|u\|_{L^\infty(\Omega)}^{-1}
\]
and, integrating between $0$ and $\|u\|_{L^{\infty}(\Omega)}$, 
\[
W_{n-1}(\Omega) \ge n \omega_n^{1+\frac{1}{n}}
\lambda^{-\frac{1}{n}}|\Omega|^{-\frac{1}{n}},
\]
that is, being $\lambda=\lambda_n(\Omega)\le
(1+\eps)\lambda_n(\Omega^*_{n-1})$,
\begin{equation}
  \label{misbas}
|\Omega|^{\frac 1 n} \ge n \omega_n^{1+\frac 1 n} W_{n-1}(\Omega)^{-1}
\lambda_{n}(\Omega^*_{n-1})^{-\frac{1}{n}} (1+\eps)^{-\frac 1 n}.
\end{equation}
As matter of fact, being $W_{n-1}(\Omega)=W_{n-1}(\Omega^*_{n-1})$,
properties \eqref{risc_palla} and \eqref{querball} give that
\[
\lambda_n(\Omega^*_{n-1})= \left(
  \frac{W_{n-1}(\Omega)}{\omega_n}\right)^{-2n} \lambda_n(B),
\]
where $B=\{|x|<1\}$. Then \eqref{misbas} becomes
\[
|\Omega|^{\frac 1 n} \ge n \omega_n^{\frac 1 n - 1} W_{n-1}(\Omega)
\lambda_n(B)^{-\frac 1 n} (1+\eps)^{-\frac 1 n},
\]
and this concludes the proof.
\end{proof}
Now we are in position to prove the main theorem of this section.
\begin{proof}[Proof of the Theorem \ref{main}]
  First of all, we observe that the quotient
  \[
  \frac{W_{n-1}(K)-W_{n-1}(L)}{W_{n-1}(K)} 
  \]
  is rescaling invariant, hence we suppose that
  $W_{n-1}(\Omega)=1$. Consequently, by Lemma \ref{boundmis} and the
  Aleksandrov-Fenchel inequality, we have that there exists two
  positive constants $c_1(n)$ and $c_2(n)$, which depend only on the
  dimension, such that
  \begin{equation}
    \label{bound}
   c_1(n) \le |\Omega| \le c_2(n).
  \end{equation}
  
For $\delta$ as in Lemma \ref{lemma1}, by \eqref{eq:3} we get that
  \[
  \begin{split}
  \inf_{0\le t \le \delta}\big\{ W_{n-1}(\Omega_t)^n
  - \omega_n^{n-1}|\Omega_t|\big\} &\le \frac {n+1} {\delta^{n+1}}
    \int_0^{\delta} t^n \big( W_{n-1}(\Omega_t)^n
    - \omega_n^{n-1}|\Omega_t|\big) dt \\
    &\le \omega_{n}^{n-1} \frac {\eps}{\delta^{n+1}} =
    \omega_{n}^{n-1} \sqrt\eps,
  \end{split}
  \]
  where we finally choose $\delta^{n+1}=\sqrt{\eps}$. Hence, this
  gives that there exists $0\le \tau \le \delta$ such that
  \begin{equation}
    \label{eq:4}
     W_{n-1}(\Omega_\tau)^n \le \omega_n^{n-1}|\Omega_\tau| +
     \omega_{n}^{n-1} \sqrt\eps.
   \end{equation}
   \fbox{\bf Case $n=2$.} In such a case, \eqref{eq:4} becomes
   \begin{equation*}
     \label{eq:7}
     \frac{P(\Omega_\tau)^2}{4\pi}- |\Omega_\tau| \le \sqrt \eps.
   \end{equation*}
   Then, denoting by $r_\tau$ and $R_\tau$ the inradius and the
   circumradius of $\Omega_\tau$ respectively, and by $\rho_\tau$ the
   radius of 
   $(\Omega_\tau)^*_1$, using the Bonnesen inequality (see for
   example \cite{oss79}, and \cite{afn09,afn11} for some related
   questions) we have
   \[
   (\rho_\tau-r_\tau)^2 \le (R_\tau-r_\tau)^2 \le \sqrt\eps.
   \]
   Being $2\pi \rho_\tau = P(\Omega_\tau) $, we have by Lemma
   \ref{lemma1}, for $\eps$ sufficiently small, that  
   \[
   r_\tau\ge \frac{ P(\Omega_\tau) }{2\pi} -\sqrt[4]{\eps}\ge
   \frac{P(\Omega)}{2\pi}\left( 1- 2\sqrt[6]{\eps}|\Omega|^{\frac 1 3}
   \right)-\sqrt[4]{\eps}\ge R(1-C_{2}|\Omega|^{\frac
     1 3} \sqrt[6]{\eps}),
   \]
   where $R=\frac{P(\Omega)}{2\pi}$ is the radius of $\Omega^*_1$ and
   $C_{2}$ denotes a constant which depends only on the dimension
   $n=2$. Being $r_\tau < r_\Omega$, by \eqref{bound} we have that 
   \begin{equation}
     \label{eq:6}
     d_2(\Omega) \le 1-
     \frac{r_\tau}{R} \le C_{2}|\Omega|^{\frac
       1 3} \sqrt[6] \eps\le C_{2} \sqrt[6]{\eps}
   \end{equation}
   where $B_{r_\tau}$ is a ball of radius $r_\tau$ contained in
   $\Omega$. Then, by \eqref{eq:6} and being $P(\Omega)=2$, we have
   that 
\begin{equation}
  \label{eq:8}
  \left( \frac{P^2(\Omega)}{4\pi} -|\Omega| \right)^{\frac 1 2}\le
  \left( \frac{\big(P(B_{r_\Omega})+C_2\sqrt[6]{\eps}\big)^2}{4\pi} - 
  |B_{r_\Omega}| \right)^{\frac 1 2} \le C_2  \sqrt[12]\eps,
\end{equation}
where last inequality follows being $P(B_{r_\Omega})\le P(\Omega)=2$.
Then, \eqref{eq:8}, \eqref{bound} and \eqref{gshaus} give
\[
\delta_H(\Omega) \le C_2 \sqrt[15] \eps.
\]
On the other hand, applying to \eqref{eq:8} the well-known Bonnesen
inequality, we get that
   \[
     D_2(\Omega) \le 2\pi(R_\Omega-r_\Omega)\le C_2 \sqrt[12] \eps.
   \] 
\fbox{\bf Case $n> 2$.}
From \eqref{eq:4} and the Aleksandrov-Fenchel inequalities
\eqref{af-2} with $k=n$, we have that
\[
\frac{P(\Omega_\tau)^{\frac{n}{n-1}}}{(n^n\omega_n)^{\frac{1}{n-1}}}\le
|\Omega_\tau|+\sqrt \eps.
\]
Hence, by \eqref{bound} and for $\eps$ sufficiently small, an
elementary inequality gives that
\[
\frac{P(\Omega_\tau)^{{n}}}{n^n\omega_n}\le
|\Omega_\tau|^{n-1}+C_n\sqrt \eps.
\]
Then, applying the stability result \eqref{gs}, and using again
\eqref{bound}, it follows that
\begin{equation}
  \label{gs2}
  (R_\tau-r_\tau)^{(n+3)/2} \le C_n
  \frac{P(\Omega)^{(n^2-3)/2}}{|\Omega|^{(n+3)(n-2)/2}} \sqrt\eps\le
  C_n\sqrt \eps,  
\end{equation}
where, as before, $R_\tau$ and $r_\tau$ are, respectively, the
circumradius and the inradius of $\Omega_\tau$.

Now, let $\rho_\tau$ be the radius of the ball
$(\Omega_\tau)^*_{n-1}$, having the same $W_{n-1}$ measure of
$\Omega_\tau$. Similarly as before, being $\rho_\tau <R_\tau$, by
\eqref{gs2}, Lemma \ref{lemma1} and \eqref{bound}, for $\eps$
sufficiently small we have
\begin{multline}
  \label{catenella}
  r_\tau \ge \rho_\tau - C_n \eps^{\frac{1}{n+3}} =
  \frac{W_{n-1}(\Omega_\tau)}{\omega_n} - C_n
  \eps^{\frac{1}{n+3}} \ge \\ \ge
  \omega_n^{-1}W_{n-1}(\Omega)\left(1-2\eps^{1/(2n+2)}|\Omega|^{\frac{1}{n+1}}
  \right) - C_n \eps^{\frac{1}{n+3}} \ge  \\ \ge
  R\left(1- C_n \eps^{1/(2n+2)}
  \right),
\end{multline}
where $R=\omega_n^{-1}W_{n-1}(\Omega)$ is the radius of the ball
$\Omega^*_{n-1}$.
Denoting again with $r_\Omega$ the inradius of $\Omega$, we have that
$r_\tau \le r_\Omega$ and
\begin{equation*}
  %\label{fine1}
  d_n(\Omega) \le 1- \frac{r_\tau}{R} \le C_n \eps^{1/(2n+2)}.
\end{equation*}
As matter of fact, by the Aleksandrov-Fenchel inequalities,
\eqref{catenella} and being $W_{n-1}(\Omega)=1$, it follows that
\begin{equation}\label{catena} 
  \begin{split}   
   \left[\left(
        \frac{P(\Omega)}{n\omega_n} \right)^n - \left( \frac{ |\Omega|
        }{\omega_n}\right)^{n-1}\right]^{\frac{2}{n+3}}
    &\le 
    \left[
      \left( \frac{W_{n-1}(\Omega)^n}{\omega_n^n}
      \right)^{n-1} - 
      \left( \frac{ |\Omega|
        }{\omega_n}\right)^{n-1}\right]^{\frac{2}{n+3}} \le \\ & \le
    \left[
      \left( \frac{W_{n-1}(B_{r_\Omega})+C_n \eps^{\frac{1}{2n+2}}}{\omega_n}
      \right)^{n(n-1)}-  
      \left( \frac{ |B_{r_\Omega}| 
        }{\omega_n}\right)^{n-1}\right]^{\frac{2}{n+3}} \le \\ & \le 
    \left[
      \left( \frac{W_{n-1}(B_{r_\Omega})}{\omega_n}
      \right)^{n(n-1)} + C_n \eps^{\frac{1}{2n+2}} -
      \left( \frac{ |B_{r_\Omega}| 
        }{\omega_n}\right)^{n-1}\right]^{\frac{2}{n+3}} = \\ &=
    C_n \eps^{\frac{1}{(n+1)(n+3)}}.
  \end{split}
\end{equation}
Hence, applying \eqref{gshaus}, from \eqref{catena}, \eqref{bound} and being
$W_{n-1}(\Omega)=\omega_n R=1$, we get that
\begin{equation*}
\delta_H(\Omega)\le C_n \eps^{\frac{1}{(n+1)(n+3)}}, 
\end{equation*}
while applying \eqref{gs}, we get
\begin{equation*}
D_n(\Omega) \le \omega_n(R_\Omega-r_\Omega) \le C_n
\eps^{\frac{1}{(n+1)(n+3)}}, 
\end{equation*}
and this concludes the proof.
\end{proof}
\begin{rem}
  \label{remdef}
  Under the assumption of Theorem \ref{main}, from \eqref{catena} and
  \eqref{catenella} an estimate for the deficiency of $\Omega$ holds,
  that is
  \[
  \Delta(\Omega) \le C_{n} \eps^{\frac{1}{(n+1)(n+3)}}.
  \]
\end{rem}
\section{The case of the $k$-Hessian operator, $1 \le k\le n-1$}
In this section we consider the eigenvalue problem related to the
$k$-Hessian operators, $1 \le k\le n-1$, namely
\begin{equation*}
  %\label{autsk2}
  \left\{
    \begin{array}{ll}
      \Sk_k(D^2u)=\lambda (-u)^k &\text{in } \Omega,\\
      u=0 &\text{on } \de\Omega,
    \end{array}
  \right.
\end{equation*}
obtaining the stability result as follows.
\begin{theo}
  \label{main2}
  Let $1\le k \le n-1$, and  $\Omega \subset \R^n$ be as in
  \eqref{ipomega} such that 
  \begin{equation}
    \label{small2}
    \lambda_k(\Omega) \le (1+\eps)\lambda_k(\Omega_{k-1}^*), 
  \end{equation}
  where $\eps>0$ is sufficiently small and $\Omega_{k-1}^*$ is the
  ball   such that $W_{k-1}(\Omega)=W_{k-1}(\Omega_{k-1}^*)$.
  Moreover we suppose that the eigenfunctions related to
  $\lambda_k(\Omega)$ have convex level sets. Then,
  \[
  \delta_H(\Omega) \le C_{n,k}\eps^{\frac{2\alpha}{n+3}},
  \]
and
  \begin{equation}
    \label{tesik}
    d_k(\Omega)\le
    C_{n,k} \eps^{\alpha},\quad D_k(\Omega) \le
    C_{n,k} \eps^{\frac{2\alpha}{n+3}},
  \end{equation}
  where $\alpha=\max\left\{\frac{1}{k+1}, \frac
    {2k}{(k+1)(n+3)}\right\}$, $C_{n,k}$ is a positive constant which
  depends only on $n$ and $k$, and $d_k(\Omega)$ and $D_k(\Omega)$
  are, respectively, the interior and exterior $k$-deficiency of
  $\Omega$ as in \eqref{defe}. 
\end{theo}
\begin{rem}
As observed in Section 2.3, the additional hypothesis on the convexity
of the level sets of the eigenfunctions corresponding to
$\lambda_k(\Omega)$ is necessary to have that a Faber-Krahn inequality
holds. On the other hand this assumption seems to be natural.  Indeed,
for $k=1$ this is due to the Korevaar concavity maximum principle
(see\cite{kor}), while it is trivial for $k=n$. For the $k$-Hessian
operators, at least in the case $n=3$ and $k=2$, it in \cite{lmx10}
and \cite{sa12} is proved that if $\Omega$ is sufficiently smooth, the
eigenfunctions of $\Sk_2$ have convex level sets. Up to our knowledge,
the general case is an open problem. 
\end{rem}
\begin{rem}
 Similarly as observed in Remark \ref{remeq}, by the estimates
 \eqref{tesik} we can obtain that 
  \begin{equation*}
    \frac{W_{k-1}(\Omega) - W_{k-1}(B_{r_\Omega})}{W_{k-1}(\Omega)} \le
    C_{n,k} \eps^{\alpha},\quad \frac{W_{k-1}(B_{R_\Omega}) -
      W_{k-1}(\Omega)}{W_{k-1}(\Omega)} \le
    C_{n,k} \eps^{\frac{2\alpha}{n+3}}.
  \end{equation*}
\end{rem}

Similarly to the case of the Monge-Amp\`ere operator, to give the
proof of Theorem \ref{main2} we first consider some preliminary
results. 

Using the same notations of section 3, for $\delta\ge 0$, we denote 
\[
\Omega_\delta = \{x\in \Omega \colon -u > \delta \}. 
\]

\begin{lemma}
  \label{lemma1kkk}
Under the hypotheses of Theorem \ref{main2}, if $u$ is the 
eigenfunction of $\Sk_k$ in $\Omega$ such that
$\|u\|_{L^{k+1}(\Omega)}=1$, then for any $\delta$ such that
$0<\delta<\frac 1 2 |\Omega|^{-\frac{1}{k+1}}$, we have 
 \begin{equation}
   \label{lemma1k}
    W_{k-1}(\Omega_\delta) \ge
 W_{k-1}(\Omega) \left[1-(n-k+1)\max\{\eps, 2\delta
   |\Omega|^{\frac{1}{k+1}}\}\right].
 \end{equation}
\end{lemma}
\begin{proof}
For $\delta>0$, we compute the Rayleigh quotient of the function
$\phi=u+\delta$ in $\Omega_\delta$. Then,
    \begin{equation}
      \label{udeltak}
          \lambda_k(\Omega_\delta) \le \dfrac{\int_{\Omega_\delta}
      (-\phi)\Sk_k (D^2 \phi)\,dx }{ \int_{\Omega_\delta}
      (-\phi)^{k+1}\,dx} = \dfrac{\int_{\Omega_\delta}
      (-u-\delta)\Sk_k (D^2 u) \,dx }{ \int_{\Omega_\delta}
     (-u-\delta)^{k+1}\,dx}. 
 \end{equation}
   \[
   \begin{split}
     %\label{eq:1}
     \int_{\Omega_\delta}(-u-\delta)\Sk_k (D^2u) \,dx
     &= \lambda_k(\Omega) \int_{\Omega_\delta} (-u-\delta)(-u)^k\,dx \\
     &\le \lambda_k(\Omega) \left(\int_{\Omega_\delta}
       (-u-\delta)^{k+1}\,dx\right)^{\frac{1}{k+1}}
     \left(\int_{\Omega_\delta}
       (-u)^{k+1}\,dx\right)^{\frac{k}{k+1}} \\
     &\le \lambda_k(\Omega) \left(\int_{\Omega_\delta}
       (-u-\delta)^{k+1}\,dx\right)^{\frac{1}{k+1}}. 
   \end{split}
   \]
   Hence, combining the above estimate with \eqref{udeltak} it follows
   that
   \begin{equation}
     \label{udelta2k}
   \lambda_k(\Omega_\delta) \le \lambda_k(\Omega) \left(\int_{\Omega_\delta}
       (-u-\delta)^{k+1}\,dx \right)^{-\frac{k}{k+1}}    
   \end{equation}
   On the other hand, by Minkowski inequality and choosing
   $\delta<\frac 1 2 |\Omega|^{-\frac{1}{k+1}}$, we obtain
   that 
   \[
   \begin{split}
   \left(\int_{\Omega_\delta} (-u-\delta)^{k+1}\,dx\right)^{\frac{1}{k+1}}
   &\ge \left( \int_{\Omega_\delta} (-u)^{k+1} \,dx \right)^{\frac
     {1}{k+1}} - \left( \int_{\Omega_\delta} \delta^{k+1} \,dx \right)^{\frac
     {1}{k+1}} \\ &\ge \left(1- \int_{\Omega\setminus\Omega_\delta}
     \delta^{k+1} \,dx \right)^{\frac{1}{k+1}}- \delta
   |\Omega_\delta|^{\frac{1}{k+1}} \\
   &= \left(1-
     \delta^{k+1} \big(|\Omega|-|\Omega_\delta|\big)
   \right)^{\frac{1}{k+1}}- \delta |\Omega_\delta|^{\frac{1}{k+1}} \\
   &\ge 1-\delta\big( |\Omega| - |\Omega_\delta| \big)^{\frac{1}{k+1}}
   - \delta |\Omega_\delta|^{\frac{1}{k+1}} \ge 1-2\delta
   |\Omega|^{\frac {1}{k+1}}.
   \end{split}
   \]
   Hence, from \eqref{udelta2k}, \eqref{small2} and Faber-Krahn
   inequality it follows that 
   \[
  \lambda_k((\Omega_\delta)^*_{k-1}) \le \lambda_k(\Omega_\delta) \le
  (1+\eps) \lambda_k(\Omega^*_{k-1}) \big( 1-2\delta |\Omega|^{\frac
    {1}{k+1}} \big)^{-k} 
  \]
  which implies, by \eqref{risc_palla}, that
  \begin{equation}
    \label{udelta3k}
    \left(\frac{W_{k-1}(\Omega_\delta)}{W_{k-1}(\Omega)}
    \right)^{\frac{2k}{n-k+1}} 
    =
    \frac{\lambda_k(\Omega^*_{k-1})}{\lambda_k((\Omega_\delta)^*_{k-1})}  
    \ge \frac{\big( 1-2\delta |\Omega|^{\frac{1}{k+1}}
      \big)^k}{1+\eps}, 
  \end{equation}
  where we used that the balls $\Omega^*_{k-1}$ and
 $(\Omega_\delta)^*_{k-1}$ preserve, respectively, the $(k-1)$-th
 quermassintegral of $\Omega$ and $\Omega_\delta$. Hence, by
 \eqref{udelta3k} we get that
 \[
 \frac{W_{k-1}(\Omega_\delta)}{W_{k-1}(\Omega)}
 \ge \left( 1 - \frac{\eps+ 2 \delta |\Omega|^{\frac
       {1}{k+1}}}{1+\eps} \right)^{\frac{n-k+1}{2}} \ge
 1-(n-k+1)\max\{\eps, 2\delta |\Omega|^{\frac{1}{k+1}}\},
 \]
 obtaining the thesis.
\end{proof}
\begin{lemma}
  Under the hypotheses of Theorem \ref{main2}, if $u$ is the 
eigenfunction of $\Sk_k$ in $\Omega$ such that
$\|u\|_{L^{k+1}(\Omega)}=1$, we have that 
  \begin{equation}
    \label{eq:10}
    \frac {n(n-k+1)^{k}}{k}  \int_0^{\max|u|}
    \frac{W_k(\Omega_t)^{k+1}-W_k((\Omega_t)^*_{k-1})^{k+1}}{\big[-\frac
      {d}{dt} W_{k-1}(\Omega_t)\big]^k} 
    dt \le \eps \lambda_k(\Omega_{k-1}^*).
  \end{equation}
\end{lemma}
\begin{proof}
The divergence form of $\Sk_k$ and the coarea formula give that
  (see also \cite{trudiso})  
  \[
  \lambda_k(\Omega) = \frac 1 k \int_0^{\max(-u)} dt \int_{\Sigma_t}
  H_{k-1}(\Sigma_t) |Du|^k d\mathcal H^{n-1},
  \]
  where $\Sigma_t=\partial \Omega_t$. Then, the H\"older inequality
  and the Reilly formula \eqref{reillyder} give that
  \begin{multline}
    \label{lemmak}
     \lambda_k(\Omega) \ge \frac 1 k \int_0^{\max(-u)}
  \dfrac{ \left( \int_{\Sigma_t} H_{k-1}(\Sigma_t) d\mathcal
      H^{n-1}\right)^{k+1}}
  {\left(
      \int_{\Sigma_t} \frac{H_{k-1}(\Sigma_t)}{|Du|}d\mathcal H^{n-1}
    \right)^{k} } dt = \\ =
  \frac {n(n-k+1)^{k}}{k} \int_0^{\max(-u)}
   \frac{ W_k(\Omega_t)^{k+1} }
  {\big[-\frac{d}{dt} W_{k-1}(\Omega_t)\big]^{k} } dt.
\end{multline}
Moreover, being $\|u^*_{k-1}\|_{k+1} \ge \|u\|_{k+1}=1$, we have
\begin{multline}
  \label{lemmakball}
\lambda_k(\Omega^*_{k-1}) \le
\frac{ \int_{\Omega^*_{k-1}} (-u^*_{k-1})\Sk_k(D^2 u^*_{k-1}) dx }
{\int_{\Omega^*_{k-1}} (-u^*_{k-1})^{k+1} dx } \le
\int_{\Omega^*_{k-1}} (-u^*_{k-1})\Sk_k(D^2 u^*_{k-1}) dx = \\ =
 \frac {n(n-k+1)^{k}}{k} \int_0^{\max(-u)}
   \frac{ W_k((\Omega_t)^*_{k-1})^{k+1}}
  {\big[-\frac{d}{dt} W_{k-1}(\Omega_t)\big]^{k} } dt.
\end{multline}
Last equality follows from the symmetry of $u^*_{k-1}$ and being
$W_{k-1}(\Omega_t)= W_{k-1}\big((\Omega_t)^*_{k-1}\big)$. Hence, taking
\eqref{lemmak} and \eqref{lemmakball} and subtracting, from we have that
\begin{multline*}
\eps \lambda_k(\Omega^*_{k-1}) \ge \\ \ge \lambda_k(\Omega)
-\lambda_k(\Omega^*_{k-1}) \ge
\frac {n(n-k+1)^{k}}{k} \int_0^{\max(-u)}
   \frac{  W_k(\Omega_t)^{k+1}- W_k((\Omega_t)^*_{k-1})^{k+1}}
  {\big[-\frac{d}{dt} W_{k-1}(\Omega_t)\big]^{k} } dt, 
\end{multline*}
that gives the thesis.
\end{proof}
In the next result we prove a lower bound for $|\Omega|$ in term of $W_{k-1}(\Omega)$. 
\begin{lemma}
\label{boundmisk}
Under the hypotheses of Theorem \ref{main2}, it holds that 
\begin{equation*}
%\label{mis}
|\Omega| \ge C_{n,k} W_{k-1}^{\frac{n}{n-k+1}}(\Omega),
\end{equation*}
where $C_{n,k}$ denotes a positive constant depending only on $n$ and $k$.
\end{lemma}
\begin{proof}
Let be $u$ an eigenfunction corresponding to the eigenvalue
$\lambda=\lambda_k(\Omega)$ and such that $\|u\|_{L^{k+1}}=1$. Then,
\begin{equation}
\label{eqk}
\Sk_k (D^2u) = \lambda (-u)^k \quad\text{ in }\Omega.
\end{equation} Arguing as in Lemma \ref{boundmis}, by
\eqref{introintcurv} 
\eqref{reillyder} and the H\"older inequality we have
\begin{equation}
\label{sinistrak}
\int_{\Omega_t}\Sk_k (D^2u)dx = \frac{1}{k} \int_{\Sigma_t} H_{k-1} |Du|^k
d\mathcal H^{n-1} \ge C_{n,k} \frac{\left(W_k(\Omega_t)\right)^{k+1}}
{\left(-\frac{d}{dt}W_{k-1}(\Omega_t)\right)^k}.
\end{equation}
We divide the proof in three cases.

Case $k> \frac n 2$. By H\"older inequality we have:
\begin{equation}
\label{destrak}
\left(\int_{\Omega_t}(-u)^kdx\right)^{\frac 1 k} \le
|\Omega|^{\frac 1 k} \|u\|_{L^{\infty}(\Omega)}.
\end{equation}
Putting togheter \eqref{sinistrak} and \eqref{destrak}, by \eqref{eqk}
we get that
\[
W_k(\Omega_t)^{-\frac{k+1}{k}} \left(
  -\frac{d}{dt}W_{k-1}(\Omega_t) \right) \ge C_{n,k} |\Omega|^{-\frac
  1 k} 
\|u\|_{\infty}^{-1}\lambda^{-\frac 1 k}.
\]
Using the Aleksandrov-Fenchel inequalities \eqref{afineq} with $j=k$
and $i=k-1$, and integrating between
$0$ and $\|u\|_{L^{\infty}(\Omega)}$, being $k>\frac n 2$ we get  
\[
W_{k-1}(\Omega)^{\frac{2k-n}{k(n-k+1)}}\ge C_{n,k}
\lambda^{-\frac{1}{k}}|\Omega|^{-\frac{1}{k}}. 
\]
Being $\lambda_k(\Omega) \le (1+\eps)\lambda_k(\Omega^*_{k-1})$,
and recalling the properties \eqref{risc_palla} and
\eqref{querball}, we have: 
\begin{gather*}
\begin{split}
  |\Omega|^{\frac{1}{k}} &
  \ge C_{n,k} W_{k-1}(\Omega)^{-\frac{2k-n}{k(n-k+1)}}
{\lambda_{k}(\Omega^*_{k-1})^{-\frac{1}{k}}} {(1+\eps)^{-\frac{1}{k}}}
=
\\ &=
C_{n,k}   W_{k-1}(\Omega)^{-\frac{2k-n}{k(n-k+1)}}
W_{k-1}(\Omega^*_{k-1})^{\frac {2}{n-k+1}}
{\lambda_{k}(B_1)^{-\frac{1}{k}}} 
{(1+\eps)^{-\frac{1}{k}}}
= \\ &= C_{n,k}
{(1+\eps)^{-\frac{1}{k}}}W_{k-1}(\Omega)^{\frac{n}{k(n-k+1)}}, 
\end{split}
\end{gather*}
and the first case is completed.

Case $k< \frac n 2$. By H\"older inequality, and being $\|u\|_{k+1}=1$
we have:
\begin{equation}
\label{destrak2}
\left(\int_{\Omega_t}(-u)^kdx\right)^{\frac 1 k} \le
|\Omega_t|^{\frac {1}{k(k+1)}} \left( \int_{\Omega_t} u^{k+1}dx
  \right)^{\frac {1}{k+1}}\le |\Omega|^{\frac {1}{k(k+1)}}.
\end{equation}
Then, joining \eqref{sinistrak} and \eqref{destrak2}, and using the
Aleksandrov-Fenchel inequalities we get
\[
W_{k-1}(\Omega_t)^{-\frac{(k+1)(n-k)}{k(n-k+1)}} \left(
  -\frac{d}{dt}W_{k-1}(\Omega_t) \right) \ge C_{n,k} |\Omega|^{-\frac
  {1} {k(k+1)}}\lambda^{-\frac 1 k}.
\]
Integrating between $0$ and $\delta$ sufficiently small, we get that
\[
W_{k-1}(\Omega_\delta)^{-\frac{n-2k}{k(n-k+1)}}-
W_{k-1}(\Omega)^{-\frac{n-2k}{k(n-k+1)}} \ge C_{n,k} |\Omega|^{-\frac
  {1} {k(k+1)}}\lambda^{-\frac 1 k} \delta. 
\]
Now we apply Lemma \ref{lemma1kkk}. Let $\eps$ and $\delta$ such that 
$\eps<2\delta|\Omega|^{\frac{1}{k+1}}<(n-k+1)^{-1}$. Hence, writing
$\alpha=-\frac{n-2k}{k(n-k+1)}<0$, we get
\begin{equation}
  \label{boh1}
  W_{k-1}(\Omega)^{\alpha}\left[
    (1-2\delta (n-k+1) |\Omega|^{\frac{1}{k+1}}
    )^\alpha-1 \right]
  \ge C_{n,k} |\Omega|^{-\frac{1} {k(k+1)}}\lambda^{-\frac 1
    k}\delta.
\end{equation}
Moreover, if $\delta$ is such that $2\delta
|\Omega|^{\frac{1}{k+1}}(n-k+1)\le 1-{2^{-\frac{1}{1-\alpha}}}$, from
\eqref{boh1} we get
\[
  W_{k-1}(\Omega)^{\alpha}\left[
    -4\alpha  (n-k+1) |\Omega|^{\frac{1}{k+1}}
     \right] \delta
  \ge C_{n,k} |\Omega|^{-\frac{1} {k(k+1)}}\lambda^{-\frac 1
    k}\delta,
  \]
  that is
  \[
  |\Omega|^{\frac{1}{k}}
  \ge C_{n,k} W_{k-1}(\Omega)^{\frac{n-2k}{k(n-k+1)}} \lambda^{-\frac
    1 k}= C_{n,k} (1+\eps)^{-\frac{1}{k}} \lambda_k(B_1)^{-\frac 1 k}
  W_{k-1}(\Omega)^{\frac{n}{k(n-k+1)}}, 
  \]
  that is the thesis.

  Case $k= \frac n 2$.
  Arguing as before, we get
  \[
 \log\left(\frac{W_{\frac n 2-1}(\Omega)}
   {W_{\frac n 2-1}(\Omega_\delta)}\right) \ge C_{n} |\Omega|^{-\frac
  {4} {n (n+2)}}\lambda^{-\frac 2 n} \delta.
\]
By Lemma \ref{lemma1kkk}, it follows that if
$\eps<2\delta|\Omega|^{\frac{2}{n+2}}<\frac{2}{n+2}$,
\[
-\log\left( 1-\delta\left(n +2\right)|\Omega|^{\frac{2}{n+2}} \right)
\ge C_{n} |\Omega|^{-\frac
  {4} {n (n+2)}}\lambda^{-\frac 2 n} \delta.
\]
Then, for $\delta$ such that $\delta (n+2)
|\Omega|^{\frac{2}{n+2}}<\frac {1} {2(n+2)}$,
\[
2(n+2)|\Omega|^{\frac{2}{n+2}}\delta \ge
C_{n}|\Omega|^{-\frac{4}{n(n+2)}}\lambda^{-\frac 2 n} \delta.
\]
Then, similarly as before,
\[
|\Omega|^{\frac{2}{n}} \ge C_{n}
W_{\frac n 2 -1}(\Omega)^{\frac{4}{(n+2)}}, 
\]
and the proof of the Lemma is completed.
\end{proof}

Now we can prove the main theorem of this section.
\begin{proof}[Proof of the Theorem \ref{main2}]
  Without loss of generality, we may suppose that $W_{k-1}(\Omega)=1$.
  Indeed, the quotient
  \[
  \frac{W_{k-1}(K)-W_{k-1}(L)}{W_{k-1}(K)} 
  \]
  is rescaling invariant. Consequently, by Lemma \ref{boundmisk} and the
  Aleksandrov-Fenchel inequality, we have that there exist two
  positive constants $c_1(n,k)$ and $c_2(n,k)$, which depend only on $n$
  and $k$, such that
  \begin{equation}
    \label{boundk}
   c_1(n,k) \le |\Omega| \le c_2(n,k).
  \end{equation}
  The H\"older inequality gives that
  \begin{multline}
    \label{eq:holder}
   \eps= \delta^{k+1} = \left( \int_0^\delta dt \right)^{k+1} \le
   \left(\int_0^\delta \big[-\frac{d}{dt} W_{k-1}(\Omega_t)\big] dt\right)^k
   \int_0^\delta \frac{1}{\big[-\frac{d}{dt}
     W_{k-1}(\Omega_t)\big]^k}dt = \\ = 
   \left[W_{k-1}(\Omega)-W_{k-1}(\Omega_\delta)\right]^k
    \int_0^\delta \frac{1}{\big[-\frac{d}{dt} W_{k-1}(\Omega_t)\big]^{k}}dt.
  \end{multline}
  Hence, for $\eps>0$ sufficiently small, $\delta$ verifies the
  hypothesis in Lemma \ref{lemma1kkk}, and the inequalities
  \eqref{eq:holder}, \eqref{lemma1k}, \eqref{eq:10} and \eqref{boundk}
  imply that  
  \begin{multline*}
    %\label{eq:11}
    \inf_{t\in [0,\delta]} \left( W_k(\Omega_t)^{k+1}-
      W_k((\Omega_t)^*_{k-1})^{k+1}\right) \le \\
    \le \frac 1 \eps
    \left[W_{k-1}(\Omega)-W_{k-1}(\Omega_\delta)\right]^k 
    \int_0^{\max(-u)}
   \frac{  W_k(\Omega_t)^{k+1}- W_k((\Omega_t)^*_{k-1})^{k+1}}
  {\big[-\frac{d}{dt} W_{k-1}(\Omega_t)\big]^{k} } dt \le \\ \le
  C_{n,k} \eps^{\frac {k}{k+1}}
  \lambda_k(\Omega^*_{k-1}).
\end{multline*}
Hence, for some $\tau \in [0,\delta]$, we have that
\[
W_k(\Omega_\tau)^{k+1}
\le W_k((\Omega_\tau)^*_{k-1})^{k+1} + C_{n,k} \eps^{\frac{k}{k+1}},
\]
being $W_{k-1}(\Omega)=1$.
Moreover, an algebraic inequality and \eqref{boundk} give that
\[
{\omega_n}^{-1} W_{k}(\Omega_\tau)^{n-k+1}
\le W_{k-1}(\Omega_\tau)^{n-k} + C_{n,k} \eps^{\frac{k}{k+1}},
\]

Applying the estimates \eqref{gs} and \eqref{boundk}, and using the
same notation of the proof of Theorem \ref{main}, we have
\begin{equation*}
  %\label{gs3}
  (R_\tau-r_\tau)^{(n+3)/2} \le 
 C_{n,k}\eps^{\frac{k}{k+1}}.
\end{equation*}
Moreover, by \eqref{lemma1k} it follows that
\begin{multline}
  \label{katenella}
  r_\tau \ge \rho_\tau - C_{n,k} \eps^{\frac{2k}{(k+1)(n+3)}} =
  \left(\frac{W_{k-1}(\Omega_\tau)}{\omega_n}\right)^{\frac{1}{n-k+1}} -
  C_{n,k} 
  \eps^{\frac{2k}{(k+1)(n+3)}} \ge \\
  \ge
 \left[ \frac{ W_{k-1}(\Omega) }{\omega_n}
  \left(1-C_{n,k}\eps^{\frac{1}{k+1}}|\Omega|^{\frac{1}{k+1}}
  \right) \right]^{\frac{1}{n-k+1}} - C_{n,k} 
\eps^{\frac{2k}{(k+1)(n+3)}} \ge  \\
\ge
\left[\frac{ W_{k-1}(\Omega) }{\omega_n}\right]^{\frac{1}{n-k+1}}
\left(
1-\tilde C_{n,k}\eps^{\frac{1}{k+1}}
-C_{n,k} \eps^{\frac{2k}{(k+1)(n+3)}}
\right)\ge
R
\left(
  1- C_{n,k}\eps^\alpha
\right),
\end{multline}
where $R=\left[
  \omega_n^{-1}W_{k-1}(\Omega)\right]^{\frac{1}{n-k+1}}$ is the
radius of $\Omega^*_{k-1}$, and $\alpha=\max\left\{\frac{1}{k+1}, \frac
  {2k}{(k+1)(n+3)}\right\}$. Hence, recalling \eqref{querball}, and
being $r_\tau\le r_\Omega$ and $W_{k-1}(\Omega)=1$, then
\begin{equation*}
  \label{kkkk}
d_k(\Omega) \le \omega_n^{\frac{1}{n-k+1}}(R-r_\tau)\le
C_{n,k} \eps^{\alpha},
\end{equation*}
that is the first estimate in \eqref{tesik}.
In order to obtain the remaining estimates of the theorem,
using the Aleksandrov-Fenchel inequalities, \eqref{katenella} and
being $W_{k-1}(\Omega)=1$, we have that 
\begin{equation}\label{catenak} 
  \begin{split}   
   \left[\left(
        \frac{P(\Omega)}{n\omega_n} \right)^n - \left( \frac{ |\Omega|
        }{\omega_n}\right)^{n-1}\right]^{\frac{2}{n+3}}
    &\le \left[
      \left( \frac{W_{k-1}(\Omega)^n}{\omega_n^n}
      \right)^{\frac{n-1}{n-k+1}} - 
      \left( \frac{ |\Omega|
        }{\omega_n}\right)^{n-1}\right]^{\frac{2}{n+3}} \le
    \\
    & \le
    \left[
      \left( \frac{W_{k-1}(B_{r_\Omega})+C_{n,k}
          \eps^{\alpha}}{\omega_n}
      \right)^{\frac{n(n-1)}{n-k+1}} - 
      \left( \frac{ |B_{r_\Omega}| 
        }{\omega_n}\right)^{n-1}\right]^{\frac{2}{n+3}} \le
    \\ & \le 
    \left[
      \left( \frac{W_{k-1}(B_{r_\Omega})}{\omega_n}
      \right)^{\frac{n(n-1)}{n-k+1}} + C_{n,k} \eps^{\alpha} - 
      \left( \frac{ |B_{r_\Omega}| 
        }{\omega_n}\right)^{n-1}\right]^{\frac{2}{n+3}} =
    \\ & = 
    C_{n,k} \eps^{\frac{2\alpha}{n+3}}.
  \end{split}
\end{equation}
Hence, \eqref{catenak} and \eqref{gshaus} imply
\[
\delta_H(\Omega)\le C_{n,k}
\eps^{\frac{2\alpha}{n+3}},
\]
while, from \eqref{gs} we get
\begin{equation*}
D_k(\Omega) \le
\omega_n^{\frac{1}{n-k+1}}(R_\Omega-r_\Omega) \le C_{n,k}
\eps^{\frac{2\alpha}{n+3}},
\end{equation*}
and this concludes the proof.
\end{proof}
\begin{rem}
  Similarly as observed in Remark \ref{remdef}, from the proof of
  Theorem \ref{main2} it is possible to obtain that
  \[
  \Delta(\Omega) \le C_{n,k} \eps^{\frac{2\alpha}{n+3}}.
  \]
\end{rem}
\bibliography{/Users/francescodellapietra/Documents/Biblioteca/library}{}
\bibliographystyle{abbrv}

\end{document}